\documentclass[12pt]{amsart}
\usepackage[normalem]{ulem}
\usepackage{xcolor}
\usepackage{graphicx}
\usepackage{amsmath}
\newtheorem{thm}{Theorem}[section]
\newtheorem{cor}[thm]{Corollary}
\newtheorem{prop}[thm]{Proposition}
\newtheorem{rem}{Remark}

\numberwithin{equation}{section}
\numberwithin{figure}{section}

\newcommand{\includegraph}[2][]{\ifnum\pdfoutput=0\includegraphics[#1]{#2.eps}\else\includegraphics[#1]{#2.pdf}\fi}

\def\f(#1,#2){f(#1,#2{,\epsilon})}
\def\g(#1,#2){g(#1,#2{,\epsilon})}
\def\h(#1,#2){h(#1,#2{,\epsilon})}
\def\k(#1,#2){k(#1,#2{,\epsilon})}
\def\fo(#1,#2){f(#1,#2{,0})}
\def\go(#1,#2){g(#1,#2{,0})}
\def\ho(#1,#2){h(#1,#2{,0})}
\def\ko(#1,#2){k(#1,#2{,0})}
\def\fa(#1,#2,#3){f(#1,#2{,#3})}
\def\ga(#1,#2,#3){g(#1,#2{,#3})}
\def\ha(#1,#2,#3){h(#1,#2{,#3})}
\def\ka(#1,#2,#3){k(#1,#2{,#3})}
\let\epsilon\varepsilon

\begin{document}

\title{The entry-exit function and geometric singular perturbation theory}

\date{October 8, 2015}

\begin{abstract}
For small $\epsilon>0$, the system
$\dot x = \epsilon$,
$\dot z = \h(x,z)z$,
with $\ho(x,0)<0$ for $x<0$ and $\ho(x,0)>0$ for $x>0$, admits solutions that approach the $x$-axis while $x<0$ and are repelled from it when $x>0$.  The limiting attraction and repulsion points are given by the well-known entry-exit function.  For $\h(x,z)z$ replaced by $\h(x,z)z^2$, we explain this phenomenon using geometric singular perturbation theory.  We also show that the linear case can be reduced to the quadratic case, and we discuss the smoothness of the return map to the line $z=z_0$, $z_0>0$, in the limit $\epsilon\to0$.
\end{abstract}

\maketitle

\begin{center} 
Peter De Maesschalck \textsuperscript{a}  and Stephen Schecter \textsuperscript{b}\\
\end{center}

\vspace{.2in}

\textsuperscript{a} \address{Department of Mathematics and Statistics, Hasselt University, B-3590 Diepenbeek, Belgium,}  \email{peter.demaesschalck@uhasselt.be} \par
\textsuperscript{b} \address{Department of Mathematics, North Carolina State University, Raleigh, NC 27695-8205, USA,}  \email{schecter@math.ncsu.edu}

\vspace{.2in}


\section{Introduction}
Consider the slow-fast planar system
\begin{align}
\dot x &= \epsilon \f(x,z), \label{eq1}\\
\dot z &= \g(x,z)z, \label{eq2}
\end{align}
with $x\in\mathbb{R}$, $z\in\mathbb{R}$,
\begin{equation}
    \label{delay-conds}
    \mbox{$\fo(x,0)>0$; \quad
    $\go(x,0)<0$ for $x<0$ and $\go(x,0)>0$ for $x>0$. }
\end{equation}

For $\epsilon=0$, the $x$-axis consists of equilibria; see Figure~\ref{fig:globalsettings}(a) below.  These equilibria are normally attracting for $x<0$ and normally repelling for $x>0$.  For $\epsilon>0$, the $x$-axis remains invariant, and the flow on it is to the right. For small $\epsilon>0$, a solution that starts at $(x_0,z_0)$, with $x_0$ negative and $z_0>0$ small, is attracted quickly toward the $x$-axis, then drifts to the right along the $x$-axis, and finally is repelled from the $x$-axis.  It reintersects the line $z=z_0$ at a point whose $x$-coordinate we denote by $p_\epsilon(x_0)$.  As $\epsilon\to0$, the  return map  $p_\epsilon(x_0)$ approaches a function $p_0(x_0)$ given implicitly by the formula
\begin{equation}
\label{inout}
\int_{x_0}^{p_0(x_0)}\frac{\go(x,0)}{\fo(x,0)}\,dx=0.
\end{equation}
 In other words, the solution does not leave the $x$-axis as soon as it becomes unstable at $x=0$; instead the solution stays near the $x$-axis until a repulsion has built up to balance the attraction that occurred before $x=0$.  The function $p_0$ is called the entry-exit \cite{benoit81} or way in-way out \cite{diener84} function.

This phenomenon, in which a solution of a slow-fast system stays near a curve of equilibria of the slow limit system after it has become unstable, and leaves at a point given by an integral like \eqref{inout}, has been called ``Pontryagin delay'' \cite{mkkr} or ``bifurcation delay'' \cite{dynbif}.  As far as we know, it was originally discovered in a  different context, in which the fast variable $z$ in \eqref{eq2} is two-dimensional and, for $\epsilon=0$, the equilibrium at $z=0$ undergoes a Hopf bifurcation as $x$ passes 0; see \cite{siskova}, which was written under the direction of Pontryagin.  In this situation, it turns out that the delay phenomenon need not occur if the system is not analytic.  See \cite{neishtadt} for a recent survey.

For the system \eqref{eq1}--\eqref{delay-conds}, Pontryagin delay and the entry-exit function are discussed in \cite{mkkr,haberman79, schecter85,dem08}.
Methods include asymptotic expansions \cite{mkkr, haberman79}, comparison to solutions constructed by separation of variables \cite{schecter85}, and direct estimation of the solution and its derivatives using the variational equation \cite{dem08}.  The last paper gives the most complete results.

Note that for the system \eqref{eq1}--\eqref{delay-conds} with $\epsilon=0$, the line of equilibria along the $x$-axis loses normal hyperbolicity at the ``turning point'' $x=0$.  The blow-up method of geometric singular perturbation theory \cite{dum-rouss96, ks01-1} is today the method of choice for understanding loss of normal hyperbolicity.  However, unless nongenericity conditions are imposed at the turning point \cite{ddm-2005}, neither spherical blow-up of the turning point nor cylindrical blow-up along the $x$-axis appears to help with this problem.   Even in the nongeneric cases where blow-up does helps, it probably does not yield optimal smoothness results.

Pontryagin delay is also encountered in the codimension-one bifurcation of slow-fast systems that gives rise to the solutions known as canards; see \cite{benoit81,diener84}.  Consider for example the system
\begin{align}
\dot x &= \epsilon \f(x,z)=\epsilon(ax+bz+\ldots), \label{canard1}\\
\dot z &= \g(x,z)=-(x+cz^2+\ldots), \label{canard2}
\end{align}
with $b$ and $c$ positive.
The omitted terms in the first equation are higher order; those in the second consist of other quadratic terms and higher-order terms.  This system is codimension-one in the context of slow-fast systems because the slow nullcline $f=0$ passes through the parabolic vertex of the the fast nullcline $\g(x,z)=0$.  For $\epsilon=0$, near the origin, the parabolic curve $\g(x,z)=0$ consists of equilibria that are attracting for $z>0$ and repelling for $z<0$.  A typical solution for small $\epsilon>0$ is shown in Figure \ref{fig:canard}.  Pontryagin delay in this context has been studied using nonstandard analysis \cite{benoit81}, asymptotic expansions \cite{mkkr}, complex analysis \cite{complex}, and blow-up \cite{dum-rouss96, ks01-1}.

\begin{figure}[htb]
\includegraph[width=2.5in]{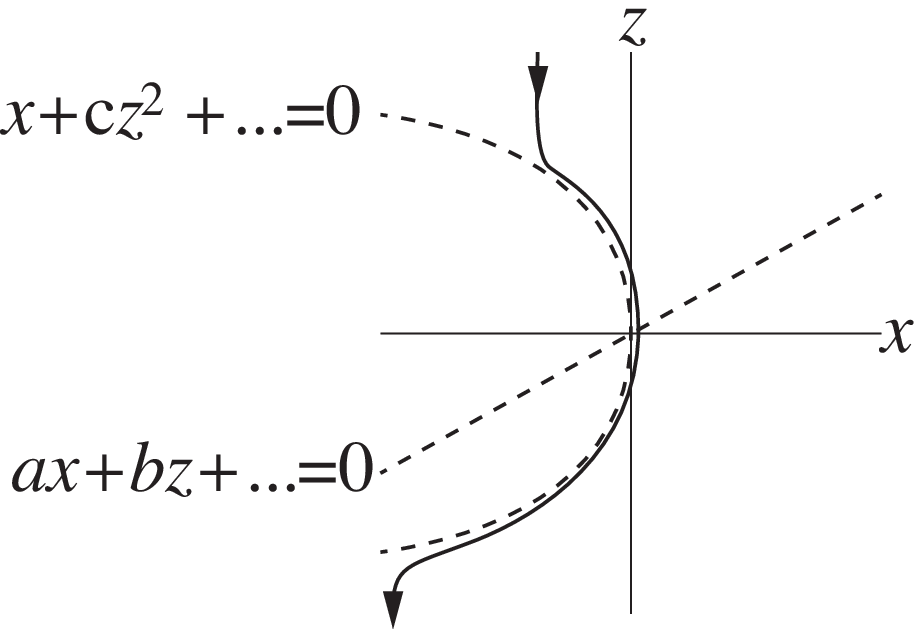}
\caption{Nullclines of \eqref{canard1}--\eqref{canard2}, and a typical solution for small $\epsilon>0$.}
\label{fig:canard}
\end{figure}

In contrast to the system  \eqref{eq1}--\eqref{delay-conds}, the system \eqref{canard1}--\eqref{canard2} for $\epsilon\neq0$ does not have  an {\em a priori} known solution near the curve of equilibria for $\epsilon=0$.  The system \eqref{canard1}--\eqref{canard2} is related to the following generalization of  \eqref{eq1}--\eqref{delay-conds}:
\begin{align}
\dot x &= \epsilon \f(x,z), \label{eq1-g}\\
\dot z &= \g(x,z)z+\epsilon \h(x,z), \label{eq2-g}
\end{align}
with \eqref{delay-conds} assumed.  For $\epsilon \neq 0$ there is in general no {\em a priori} known solution near the axis.  Like \eqref{canard1}--\eqref{canard2}, this system can be studied under generic assumptions using blow-up \cite{ddm-2005}.

In this paper we establish the entry-exit relation for \eqref{eq1}--\eqref{delay-conds} indirectly based on blow-up, without making nongeneric assumptions. The construction uses cylindrical blow-up, explains the phenomenon geometrically, and yields $C^{\infty}$-smoothness of the return map.  In contrast, the paper \cite{ddm-2005} yields smoothness in terms of some root of $\epsilon$ and $\epsilon\log\epsilon$ for the cases it treats.

We first consider, instead of \eqref{eq1}--\eqref{delay-conds}, the apparently more degenerate problem
\begin{align}
\dot x &= \epsilon \f(x,z), \label{eq3}\\
\dot z &= \g(x,z)z^2, \label{eq4}
\end{align}
with $f$ and $g$ satisfying \eqref{delay-conds}, which arose in the study of relaxation oscillations in the Holling-Tanner predator-prey model \cite{gms}.  For small $\epsilon>0$, a solution of \eqref{eq3}--\eqref{eq4} that starts at $(x_0,z_0)$, with $x_0$ negative and $z_0>0$, behaves just as in the first paragraph of this paper; see Figure~\ref{fig:globalsettings}(b). We show that, in contrast to the case for \eqref{eq1}--\eqref{eq2}, in the case of  \eqref{eq3}--\eqref{eq4} there is a nice geometric explanation using blow-up.
\begin{figure}
    \centerline{\includegraph{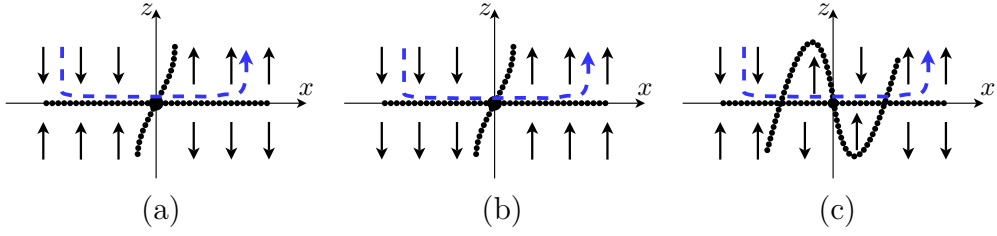}}
    \centerline{(a)\hspace{4cm}(b)\hspace{3.1cm}\qquad(c)}
    \caption{Dynamics for $\epsilon=0$ and an orbit for small $\epsilon>0$ in dashed line.  (a) For system \eqref{eq1}--\eqref{eq2}.  (b) For system \eqref{eq3}--\eqref{eq4}.  (c) Case of multiple turning points.}
    \label{fig:globalsettings}
\end{figure}

We shall say that a  function $h(x_1,\ldots,x_m,z)$ {\em has property $\mathcal{F}$ in $z$} if $h(x_1,\ldots,x_m,z)-h(x_1,\ldots,x_m,0)$ is  flat in $z$ as $z\to 0$, i.e.,
    \[
        |h(x_1,\ldots,x_m,z) - h(x_1,\ldots,x_m,0)| = \mathcal{O}(z^N) \mbox{ as } z\to 0 \mbox{ for all } N>0.
    \]
Using blow-up, we shall prove the following result.

\begin{thm}
\label{th:main}
Consider \eqref{eq3}--\eqref{eq4}, where $f$ and $g$ are $C^\infty$ and satisfy \eqref{delay-conds}.  Choose $x_0^*<0$ such that $p_0(x_0^*)$ can be defined using \eqref{inout}, i.e.,
$\int_{x_0^*}^{p_0(x_0^*)}\frac{\go(x,0)}{\fo(x,0)}\,dx=0$.  If $z_0>0$ and $\epsilon_0>0$ are sufficiently small, and $I_0$ is a sufficiently small neighborhood of $x_0^*$, then:
\begin{enumerate}
\item For $0<\epsilon<\epsilon_0$ and $x_0 \in I_0$, the solution of  \eqref{eq1}--\eqref{eq2} through $(x_0,z_0)$
first reintersects the line $z=z_0$ in a point $(x,z)=(p_\epsilon(x_0),z_0)$.
\item Define $p:I_0\times [0,\epsilon_0)\to\mathbb{R}$ by $p(x_0,\epsilon)=p_\epsilon(x_0)$.  (Thus for $\epsilon=0$, $p$ is defined using $p_0$.) Then there is a $C^\infty$ function $\tilde p$ of three variables such that $p(x_0,\epsilon)=\tilde p(x_0,\epsilon,\epsilon\log\epsilon)$.
\item If  $\g(x,z)/\f(x,z)$ has property $\mathcal{F}$ in $z$, then $p$ is a $C^\infty$ function of $(x_0,\epsilon)$.
\end{enumerate}
\end{thm}

The dependence of $p$ on $\epsilon \log \epsilon$ results from the fact that, after blow-up, solutions must pass by a line of saddle equilibria with positive and negative eigenvalues of equal magnitude.  Because of this resonance, changes of coordinates leave higher-order terms of every order.  We note that in other applications of the blow-up technique, flow past a resonant saddle leads to dependence of functions on a root of $\epsilon$ as well as on $\epsilon \log \epsilon$.
It is common for the blow-up technique to lead to flow past a resonant saddle, which sometimes results in suboptimal results.  Here, the result in Theorem~\ref{th:main} is optimal.  To demonstrate this, we will show in Section~\ref{sec:ex} that the  return map for
\begin{align}
\dot x &= \epsilon, \label{ex1}\\
\dot z &=  (x+\alpha z) z^2  \label{ex2}
\end{align}
has, for fixed small $\alpha$, finite differentiability due to the appearance of a logarithmic term in the expansion.

Conclusion (3) of Theorem \ref{th:main} can be used to treat the system \eqref{eq1}--\eqref{eq2} by reducing it to \eqref{eq3}--\eqref{eq4}.  Indeed, the change of variables $\mathbb{R}\times\mathbb{R}_+\to\mathbb{R}\times[0,1)$ given by $(x,w)\to(x,z)$ with
$$
z=\kappa(w)=\begin{cases}
e^{-\frac{1}{w}} & \mbox{if } w>0, \\
0& \mbox{if } w=0,
\end{cases}
$$
converts \eqref{eq1}--\eqref{eq2} to
\begin{align}
\dot x &= \epsilon \f(x,{\kappa(w)}), \label{eq3c}\\
\dot w &= \g(x,{\kappa(w)})w^2, \label{eq4c}
\end{align}
Note that $\f(x,{\kappa(w)})$ and $\g(x,{\kappa(w)})$ are as smooth as $f$ and $g$, and have property $\mathcal{F}$ in $w$.  Therefore the third conclusion of Theorem \ref{th:main} applies to \eqref{eq3c}--\eqref{eq4c}.  Interpreting in terms of the original system, we have

\begin{cor}\label{cor:main}
Conclusions (1) and (2) of Theorem \ref{th:main} remain true when the system \eqref{eq3}--\eqref{eq4} is replaced by the system \eqref{eq1}--\eqref{eq2}.  Moreover, $p$ is a $C^\infty$ function of $(x_0,\epsilon)$.
\end{cor}

What's more, we can relax the conditions on Corollary~\ref{cor:main} to an extent we haven't seen in the literature before.  Suppose conditions \eqref{delay-conds} are replaced by
\[
    \fo(x,0)>0; \qquad \go(x_0^*,0) < 0\text{ and }\go({p_0(x_0^*)},0) > 0,
\]
where $p_0$ is defined through \eqref{inout} (selecting the leftmost point where the integral is balanced in case multiple balancing points are possible).  Then Corollary~\ref{cor:main} remains valid.  In other words, what matters is to have attraction towards $z=0$ at the entry point and repulsion at the exit point; what lies in between can be anything.  With this in mind, we can treat passages through multiple turning points (see Figure~\ref{fig:globalsettings}(c)) and show that the exit point is given as a smooth perturbation of the leftmost point where the integral \eqref{inout} balances.

The reader may wonder whether one can reduce to \eqref{eq3}--\eqref{eq4} to \eqref{eq1}--\eqref{eq2} by a coordinate change, and thereby use the results of \cite{dem08} to understand \eqref{eq3}--\eqref{eq4}.  We do not see how to do this; the inverse of the coordinate change $z=\kappa(w)$ is not sufficiently differentiable.

The organization of the paper is as follows: we prove Theorem \ref{th:main} using blow-up in Section \ref{sec:bu}, but delay the somewhat technical part, a treatment of the flow past a line of resonant saddles, to Section \ref{sec:local}.

\begin{rem}\label{rem}
The assumptions in Theorem \ref{th:main} that $f$ and $g$ only depend on $(x,z,\epsilon)$ and are $C^\infty$ are there to simplify the proof.  The theorem remains true if $f$ and $g$ are functions of $(x,z,\epsilon,\alpha)$, where $\alpha$ is a finite-dimensional additional parameter.  In Section \ref{sec:finite} we discuss removing the $C^\infty$ assumption.
\end{rem}

\section{Blow-up\label{sec:bu}}

To prove Theorem  \ref{th:main}, we first note that in a neighborhood of the $x$-axis, we have $\f(x,z)>0$, so we can divide the system \eqref{eq3}--\eqref{eq4} by $f$, yielding
\begin{align}
\dot x &= \epsilon , \label{eq5}\\
\dot z &= \h(x,z)z^2, \label{eq6}
\end{align}
with $h=g/f$.

We extend  \eqref{eq5}--\eqref{eq6} to $xz\epsilon$-space:
\begin{align}
\dot x &= \epsilon, \label{eq3e} \\
\dot z &= \h(x,z)z^2, \label{eq4e} \\
\dot \epsilon &=0. \label{eq5e}
\end{align}
We then blow up the $x$-axis in $xz\epsilon$-space, which consists of equilibria of \eqref{eq3e}--\eqref{eq5e}, to a cylinder as follows.  Let $(x,(\bar z, \bar \epsilon), r)$ be a point of $\mathbb{R} \times S^1 \times \mathbb{R}_+$; we have $\bar z^2 + \bar \epsilon^2=1$.  The blow-up transformation is a map from $\mathbb{R} \times S^1 \times \mathbb{R}_+$ to $xz\epsilon$-space given by
\begin{align*}
x & = x, 
\\ z &= r\bar z, 
\\ \epsilon &= r \bar \epsilon.
\end{align*}
The system \eqref{eq3e}--\eqref{eq5e} pulls back to one on $\mathbb{R} \times S^1 \times \mathbb{R}_+$.  The system we shall study is this one divided by $r$.  Division by $r$ desingularizes the system on the cylinder $r=0$ but leaves it invariant.

\subsection{Polar coordinates}  The blow-up can be visualized most completely in polar coordinates, i.e., for $(x,(\bar z, \bar \epsilon), r) \in \mathbb{R} \times S^1 \times \mathbb{R}_+$, we set $\bar z =\cos\theta$ and $\bar\epsilon=\sin\theta$.  Thus we use coordinates $(x,\theta,r)$ with $\theta$ interpreted as an angle modulo $2\pi$.  In terms of the original coordinates $(x,z,\epsilon)$, we have
\begin{align*}
    x & = x, 
\\  z &= r\cos\theta, 
\\  \epsilon &= r\sin\theta. 
\end{align*}
After making the coordinate change and dividing by $r$, the system \eqref{eq3e}--\eqref{eq5e} becomes
\begin{align}
\dot x &= \sin\theta, \label{p1} \\
\dot r &= r\cos^3\theta \, \ha(x,r\cos\theta,r\sin\theta), \label{p2} \\
\dot \theta &=-\cos^2\theta\sin\theta \, \ha(x,r\cos\theta,r\sin\theta). \label{p3}
\end{align}

\begin{figure}[htb]
\includegraph[width=4in]{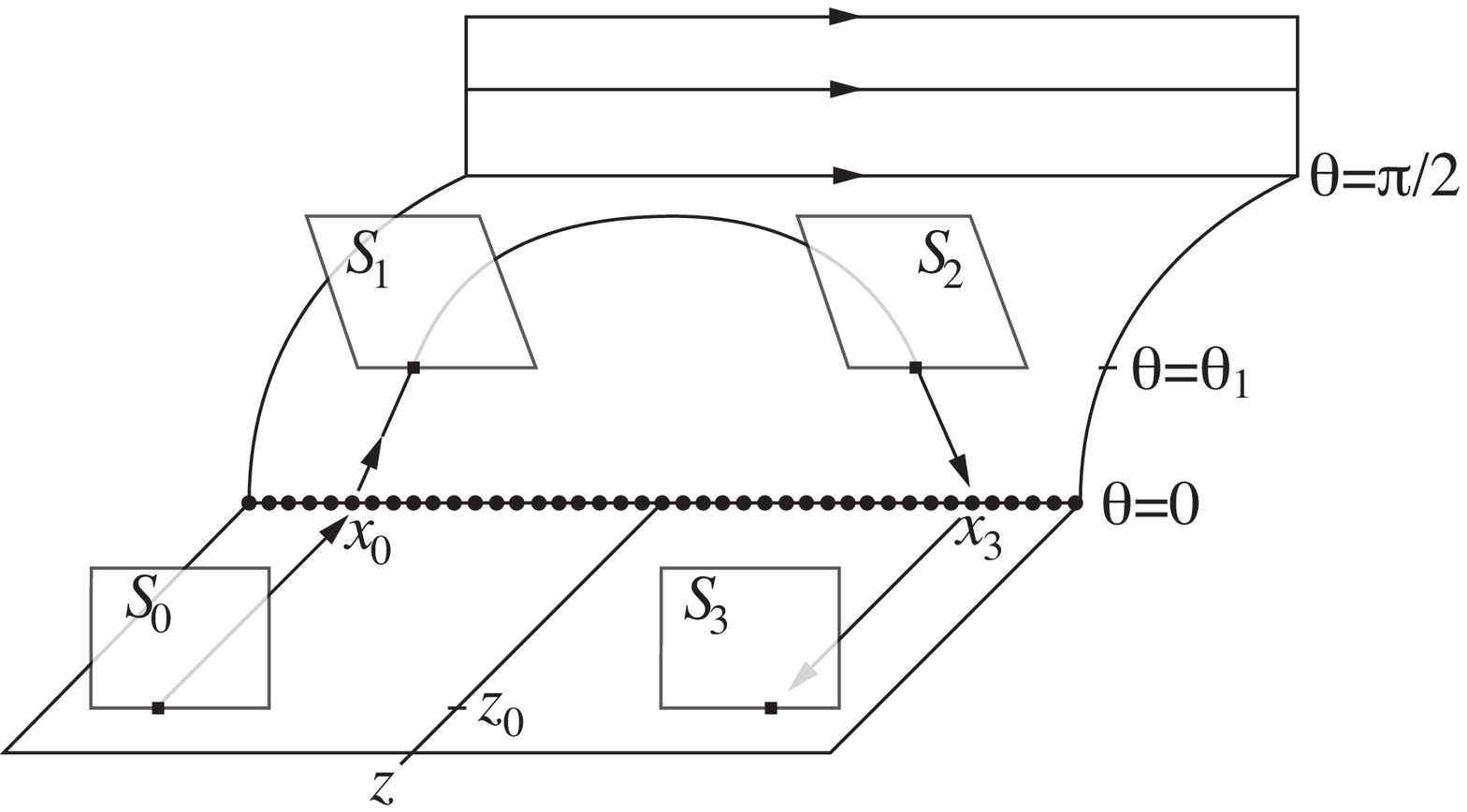}
\caption{Flow of \eqref{p1}--\eqref{p3}.}
\label{fig:flow}
\end{figure}

A portion of the flow of this system is pictured in Figure \ref{fig:flow}.
\begin{itemize}
\item The quarter cylinder is the portion of the cylinder $r=0$ between $\theta=0$ and $\theta=\frac{\pi}{2}$. The cylinder $r=0$ corresponds to the $x$-axis in $xz\epsilon$-space.  It is invariant because $r=0$ implies $\dot r=0$.  On it the system \eqref{p1}--\eqref{p3} reduces to
$$
\dot x = \sin\theta, \quad \dot \theta =-\cos^2\theta\sin\theta \, \ho(x,0).
$$
A typical solution is shown.
\item The horizontal plane is the portion of the plane $\theta=0$ with $r\ge0$.  It corresponds to the portion of the $xz$-plane in $xz\epsilon$-space with $z\ge0$.  The plane $\theta=0$ is invariant because $\theta=0$ implies $\dot\theta=0$.  On it the system \eqref{p1}--\eqref{p3} reduces to
$$
\dot x = 0, \quad \dot r = r \ho(x,r).
$$
Solutions have $x=$ constant.  Two solutions are  shown.
\item The intersection of the cylinder $r=0$ and the horizontal plane  $\theta=0$ is a line of equilibria.  For $x\neq0$ these equilibria have negative and positive eigenvalues of equal magnitude, and a zero eigenvalue.
\item The vertical plane is the portion of the plane $\theta=\frac{\pi}{2}$ with $r\ge0$.  It corresponds to the portion of the $x\epsilon$-plane in $xz\epsilon$-space with $\epsilon\ge0$.  The plane $\theta=\frac{\pi}{2}$ is invariant because $\theta=\frac{\pi}{2}$ implies $\dot\theta=0$.
On it the system \eqref{p1}--\eqref{p3} reduces to
$$
\dot x = 1, \quad \dot r =0.
$$
The solutions are lines.  For $r>0$, the line corresponds to the invariant line $z=0$, $\epsilon=r$ in $xz\epsilon$-space.
\end{itemize}

The function $p$ described in Theorem \ref{th:main} can be used to define a mapping between two rectangles in $xz\epsilon$-space.  Using the interval $I_0$ of the theorem, the domain  is
$$
R_0=\{(x,z,\epsilon) : x\in I_0, \, z=z_0, \, 0\le\epsilon<\epsilon_0\}.
$$
The codomain can be defined using a second interval $I_3$:
$$
R_3=\{(x,z,\epsilon) : x\in I_3, \, z=z_0, \, 0\le\epsilon<\epsilon_0\}.
$$
The mapping is $\hat P(x,z_0,\epsilon)=(p(x,\epsilon),z_0,\epsilon)$.

The rectangles $R_0$ and $R_3$ correspond to rectangles $S_0$ and $S_3$ in blow-up space.  They are pictured in Figure \ref{fig:flow}.

The three solutions shown in Figure \ref{fig:flow} combine to comprise a ``singular solution" from $(x,r,\theta)=(x_0,z_0,0)$ to $(x,r,\theta)=(x_3,z_0,0)$.   For small $\theta_0>0$, a solution that starts at $(x,r,\theta)=(x_0,z_0,\theta_0)\in S_0$ closely follows this singular solution until it arrives at $S_3$.  Thus we have a mapping $P$ from  $S_0$ to $S_3$  obtained by following a singular solution for $\theta_0=0$, and following an actual solution for $\theta_0>0$.  To prove Theorem \ref{th:main} it suffices to study the differentiability of $P$ and show that $x_3=p_0(x_0)$.

We study $P$ with the aid of two other rectangles $S_1$ and $S_2$ in blow-up space.  For appropriate intervals $I_1$ and $I_2$ and a fixed $\theta_1$, $0<\theta_1<\frac{\pi}{2}$,
\begin{equation}
\label{Si}
S_i = \{(x,\theta,r) : x \in I_i, \, \theta=\theta_1, \, 0\le r <\epsilon_0\csc \theta_1\}, \quad i=1,\,2.
\end{equation}
We have $P=P_3 \circ P_2 \circ P_1$ where $P_i : S_{i-1} \to S_{i}$.   For $i=1,\,3$, $P_i$ is partly defined by following a singular solution rather than a solution.

To analyze these mappings we shall use affine coordinates instead of polar coordinates,  as is customary.

\subsection{Affine coordinates for $\bar z>0$}
 For $(x,(\bar z, \bar \epsilon), r) \in \mathbb{R} \times S^1 \times \mathbb{R}_+$ with $\bar z>0$, let
$E=\frac{\bar \epsilon}{\bar z}$, and in place of $r$ use $z=r\bar z$.  Thus we have
\begin{align*}
    x & = x, 
\\  z &= z, 
\\  \epsilon &= zE, 
\end{align*}
with $z\ge0$.  Note that $E=\tan\theta$.  After division by $z$ (equivalent to division by $r$ up to multiplication by a positive function), \eqref{eq3e}--\eqref{eq5e} becomes
\begin{align}
\dot x &= E, \label{r1} \\
\dot z &= \ha(x,z,zE)z, \label{r2} \\
\dot E &=-\ha(x,z,zE)E. \label{r3}
\end{align}

 Let $E_1=\tan\theta_1>0$.  In the affine coordinates, the rectangles $S_i$ corresponds to rectangles $S_i^a$ given by
\begin{align*}
S_i^a&= \{(x,z,E) : x \in I_i, \, z=z_0, \, 0\le E <\frac{\epsilon_0}{z_0}\}, \quad i=0,\,3; \\
S_i^a&= \{(x,z,E) : x \in I_i, \, E=E_1, \, 0\le z <\frac{\epsilon_0}{E_1}\}, \quad i=1,\,2. \\
\end{align*}
See Figure \ref{fig:flow-corner}.  Looking at the orbit that connects $(x,z,E)=(x_0,0,0)$ to $(x,z,E)=(x_3,0,0)$, we see  that $0=\int_{x_0}^{x_3}\frac{dE}{dx}\,dx=-\int_{x_0}^{x_3}\ho(x,0)\,dx$, so $x_3=p_0(x_0)$.  This immediately explains the entry-exit function.

\begin{figure}[htb]
\includegraph[width=3in]{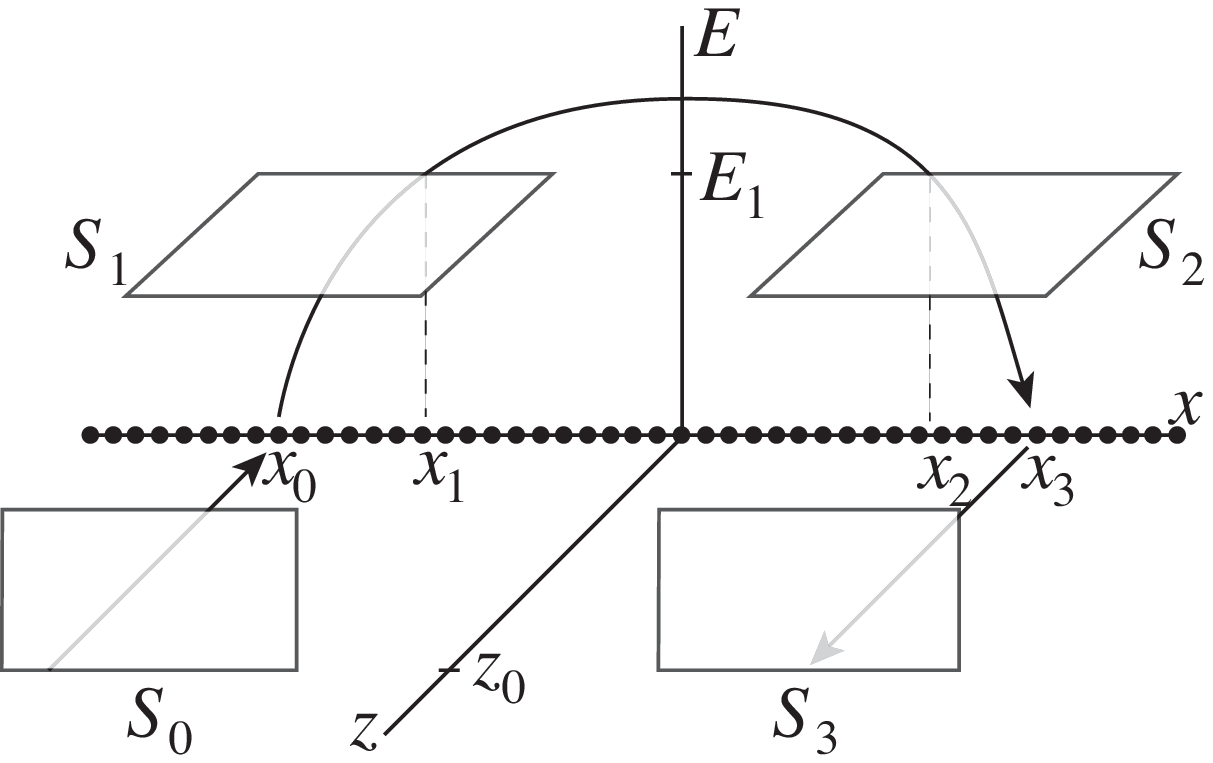}
\caption{Flow of \eqref{r1}--\eqref{r3}.}
\label{fig:flow-corner}
\end{figure}

For $i=1,2,3$, we define maps $P_i^a:S_{i-1}^a\to S_i^a$.  $P_2^a$ is defined by following solutions.  $P_1^a(x_0,z_0,E_0)$ is defined by following a solution if $E_0>0$, and by following a singular solution if $E_0=0$.  $P_2^a(x_2,z_2,E_1)$ is defined by following a solution if $z_2>0$, and by following a singular solution if $z_2=0$.

\subsection{Passage from $S_0^a$ to $S_1^a$}

We have $\ho(x_0^*,0)<0$.  To study \eqref{r1}--\eqref{r3} with $(x,z)$ near $(x_0^*,0)$, we divide by $-\h(x,z)$, let $\k(x,z)=-\frac{1}{\h(x,z)}>0$,  and obtain
\begin{align}
\dot x &= \ka(x,z,zE) E, \label{r1a} \\
\dot z &= -z, \label{r2a} \\
\dot E &=E. \label{r3a}
\end{align}
Note that $zE=\epsilon$ is constant on solutions.  If $\h(x,z)=\g(x,z)/\f(x,z)$ has property $\mathcal{F}$ in $z$, then $\k(x,z)$ has property $\mathcal{F}$ in $z$ (but $\ka(x,z,zE)$ may not).

We write  $P_1^a:S_0^a\to S_1^a$ as $P_1^a(x_0,z_0,E_0) = (x_1,z_1,E_1)$; the values of $z_0$ and $E_1$ are fixed in this formula.  For $E_0>0$ this mapping is obtained by following the solution of  \eqref{r1a}--\eqref{r3a} that starts at $(x_0,z_0,E_0)$ until it intersects the plane $E=E_1$ in a point $(x_1,z_1,E_1)$, with
$z_1=\frac{E_0}{E_1}z_0$ and $x_1=X_1(x_0,E_0)$.  For $E_0=0$ we define $z_1$ to be 0, and we define $x_1=X_1(x_0,0)$ by following the singular orbit; see Figure \ref{fig:flow-corner}.   From the normal hyperbolicity of $x$-axis away from $x=0$, it follows that the mapping $X_1(x_0,E_0)$ is continuous at $E_0=0$.

By looking at the unstable manifold of $(x_0,0,0)$ in $z=0$, we see that
\begin{equation}
\label{X1}
E_1=\int_{x_0}^{x_1}\frac{dE}{dx}\,dx = \int_{x_0}^{x_1}\frac{1}{\ko(x,0)}\,dx=- \int_{x_0}^{x_1}\ho(x,0)\,dx.
\end{equation}
This formula implicitly defines $x_1$ as a function of $x_0$, and hence implicitly defines $x_1=X_1(x_0,0)$.

\begin{prop}
\label{cornerprop1}
If $z_0>0$, $\epsilon_0>0$, and $E_1>0$ are sufficiently small, and $I_0$ is a sufficiently small neighborhood of $x_0^*$, then given $N>0$, there is a $C^N$ function $\tilde X_1$ of three variables such that
\begin{enumerate}
\item $X_1(x_0,E_0)=X_1(x_0,\frac{\epsilon}{z_0})=\tilde X_1(x_0,\epsilon,\epsilon\log \epsilon)$.
\item $X_1(x_0,0)=\tilde X_1(x_0,0,0)$ is given implicitly by \eqref{X1}.
\item If $\h(x,z)=\g(x,z)/\f(x,z)$ has property $\mathcal{F}$ in $z$, then $\tilde X_1(x,\epsilon,v) = o(v^N)$ as $v\to 0$.
\end{enumerate}
\end{prop}

 The proof of conclusions (1) and (3) of this proposition will be postponed until Section \ref{sec:local}.  Since we have already shown that
$X_1(x_0,0)$ is given implicitly by \eqref{X1}, once we know (1), it follows that $\tilde X_1(x_0,0,0)$ is also given implicitly by \eqref{X1}.

\subsection{Passage from $S_1^a$ to $S_2^a$}

$P_2^a:S_1^a\to S_2^a$ is obtained by following the solution of  \eqref{r1}--\eqref{r3} that starts at $(x_1,z_1,E_1)$ until it reintersects the plane $E=E_1$ in a point $(x_2,z_2,E_1)$, with $z_2=z_1$ and $x_2=X_2(x_1,z_1)$.   Since $C^\infty$ vector fields have $C^\infty$ flows, we have immediately that $X_2$ is $C^\infty$.

The value of $x_2=X_2(x_1,0)$ is given implicitly by the formula
\begin{equation}
\label{X2}
0=\int_{x_1}^{x_2} \frac{dE}{dx}\,dx=-\int_{x_1}^{x_2} \ho(x,0)\,dx.
\end{equation}

Define $x_1^*$ by \eqref{X1} with $x_0=x_0^*$, i.e.,
$
\int_{x_0^*}^{x_1^*} \ho(x,0)\,dx=-E_1
$.

\begin{prop}
\label{familyprop}
For a given  $E_1>0$, if $I_1$ is a sufficiently small neighborhood of $x_1^*$ and $\epsilon_0>0$ is sufficiently small, the function $x_2=X_2(x_1,z_1)$ defined above is $C^\infty$,  and $X_2(x_1,0)$ is given implicitly by \eqref{X2}.
\end{prop}

\subsection{Passage from $S_2^a$ to $S_3^a$}

Define $x_2^*$ implicitly by \eqref{X2} with $x_1=x_1^*$, i.e.,
$
 \int_{x_1^*}^{x_2^*} \ho(x,0)\,dx=0
 $.

We have $\ho(x_2^*,0)>0$.  To study \eqref{r1}--\eqref{r3} with $(x,z)$ near $(x_2^*,0)$, we divide by $\h(x,z)$, let $\k(x,z)=\frac{1}{\h(x,z)}>0$,  and obtain
\begin{align}
\dot x &= \ka(x,z,zE) E, \label{r1g} \\
\dot z &= z, \label{r2g} \\
\dot E &=-E. \label{r3g}
\end{align}

We write  $P_3^a:S_2^a\to S_3^a$ as $P_3^a(x_2,z_2,E_1) = (x_3,z_0,E_3)$; the values of $E_1$ and $z_0$ are fixed in this formula.  For $z_2>0$ this mapping is obtained by following the solution of  \eqref{r1g}--\eqref{r3g} that starts at $(x_2,z_2,E_1)$ until it intersects the plane $z=z_0$ in a point $(x_3,z_0,E_3)$, with
$E_3=\frac{E_1}{z_0}z_2$ and $x_3=X_3(x_2,z_2)$.   Note that
$$
E_3=\frac{E_1}{z_0}z_2=\frac{E_1}{z_0}z_1=\frac{E_1}{z_0}\frac{E_0}{E_1}z_0=E_0.
$$
For $z_2=0$ we define $E_3$ to be 0, and we define $x_3=X_3(x_2,0)$  by following the singular orbit; see Figure \ref{fig:flow-corner}.  By looking at the stable manifold of $(x_3,0,0)$ in $z=0$, we see that
\begin{equation}
\label{X3}
-E_1=\int_{x_2}^{x_3}\frac{dE}{dx}\,dx = -\int_{x_2}^{x_3}\frac{1}{\ko(x,0)}\,dx=- \int_{x_2}^{x_3}\ho(x,0)\,dx.
\end{equation}
This formula implicitly defines $x_3$ as a function of $x_2$, and hence implicitly defines $x_3=X_3(x_2,0)$.

The proof of Proposition~\ref{cornerprop1} that we will present in Section~\ref{sec:local} is also valid for proving the following proposition:

\begin{prop}
\label{cornerprop2}
If $z_0>0$, $\epsilon_0>0$, and $E_1>0$ are sufficiently small, and $I_2$ is a sufficiently small neighborhood of  $x_2^*$, then given $N>0$, there is a $C^N$ function $\tilde X_3$ of three variables such that
\begin{enumerate}
\item $X_3(x_2,z_2)=X_3(x_2,\frac{\epsilon}{E_1})=\tilde X_3(x_2,\epsilon,\epsilon \log \epsilon)$.
\item $X_3(x_2,0)=\tilde X_3(x_2,0,0)$ is given implicitly by \eqref{X3}.
\item If $\h(x,z)=\g(x,z)/\f(x,z)$ has property $\mathcal{F}$ in $z$, then $\tilde X_3(x,\epsilon,v) = o(v^N)$ as $v\to 0$.
\end{enumerate}
\end{prop}

\subsection{Proof of Theorem \ref{th:main}}

To prove Theorem \ref{th:main}, we define $P^a:S_0\to S_3$ by $P^a=P_3^a\circ P^a=P_3^a\circ P_2^a\circ P_1^a$.
Then $P^a(x_0,z_0,E_0)=(p^a(x_0,E_0),z_0,E_0)$ where $p(x_0,\epsilon)=p^a(x_0,\frac{\epsilon}{z_0})$.

From Propositions \ref{cornerprop1}, \ref{familyprop}, and \ref{cornerprop2}, we see that given $N>0$, there is a $C^N$ function $\tilde p_N$ of three variables such that $p(x_0,\epsilon)=\tilde p_N(x_0,\epsilon,\epsilon\log \epsilon)$.  To see that $p(x_0,0)=p_0(x_0)$, we note that $p(x_0,0)=x_3$ where by \eqref{X1}, \eqref{X2}, and \eqref{X3} we have
\begin{align*}
\int_{x_0}^{x_3}\ho(x,0)\,dx &=
\int_{x_0}^{x_1}\ho(x,0)\,dx +\int_{x_1}^{x_2}\ho(x,0)\,dx +\int_{x_2}^{x_3}\ho(x,0)\,dx \\&=-E_1+0+E_1=0.
\end{align*}

The sequence $1, \epsilon\log \epsilon,\epsilon, (\epsilon\log \epsilon)^2,\epsilon^2, \ldots$ is an {\em asymptotic scale} at $\epsilon=0$: each term divided by the previous term approaches 0 as $\epsilon$ approaches 0.  It follows that the asymptotic expansion of $p_N$ in terms of $(\epsilon,\epsilon\log\epsilon)$ coincides with the expansion of $p_{M}$ up to order $\min(N,M)$.  Hence there is a unique power series $\check{p}(x_0,\epsilon,v)$ in $(\epsilon,v)$, with coefficients that are smooth functions of  $x_0$, that is the asymptotic expansion of $p(x_0,\epsilon)$  in terms of $(\epsilon,\epsilon\log\epsilon)$. We can use Borel's theorem to realize $\check{p}(x_0,\epsilon,v)$ as a smooth function $\bar p(x_0,\epsilon,v)$.  We readily see that
\[
    p(x_0,\epsilon)-\bar p(x_0,\epsilon,\epsilon\log\epsilon)
\]
is $\mathcal{O}(\epsilon^N)$-flat for all $N$, uniformly in $x_0$ on compact sets, along with all its derivatives.  It follows that $p-\bar p$ is $C^\infty$  in $(x_0,\epsilon)$.  Thus we have written $p(x_0,\epsilon)$ as a $C^\infty$ function of $(x_0,\epsilon,\epsilon\log\epsilon)$.

If $\g(x,z)/\f(x,z)$ has property $\mathcal{F}$ in $z$, we see from conclusion (3) of Proposition \ref{cornerprop1} that $\check{p}(x_0,\epsilon,v)$ does not depend on $v$ at all.
Therefore $\bar p$ can be taken to be a $C^\infty$ function of $(x_0,\epsilon)$ only.  Hence in this case $p$ is a
$C^{\infty}$ function of $(x_0,\epsilon)$.

\section{Passage by the line of saddles}\label{sec:local}

In this Section we prove  conclusions (1) and (3) of Proposition \ref{cornerprop1}; the proof of Proposition \ref{cornerprop2} is similar.  We  shall  express $X_1(x_0,E_0)$ with $E_0>0$ in the form $X_1(x_0,E_0)=\tilde X_1(x_0,\epsilon,\epsilon\log\epsilon)$, where $\tilde X_1$ extends smoothly to $\tilde X_1(x_0,0,0)$.

\subsection{Normal form}\label{sec:nf}

To simplify \eqref{r1a}--\eqref{r3a}, we first consider the autonomous differential equation
$$
\frac{dx}{dE}=k(x,0,\epsilon),
$$
in which $E$ plays the role of time and $\epsilon$ is a parameter.  We denote the flow by $x=\alpha(\tilde x,\epsilon,E)$, i.e., $\alpha(\tilde x,\epsilon,E)$ satisfies
\begin{equation}
\label{alpha}
\alpha_E(\tilde x,\epsilon,E)=k(\alpha(\tilde x,\epsilon,E),0,\epsilon),\quad \alpha(\tilde x,\epsilon,0)=\tilde x.
\end{equation}
Since $k$ is $C^\infty$, $\alpha$ is $C^\infty$.

\begin{prop}
\label{prop:normalform1}
If in \eqref{r1a}--\eqref{r3a} one makes the change of variables $x=\alpha(\tilde x,\epsilon,E)$, with $\epsilon=zE$.  Then
\begin{equation}
\dot{\tilde x}=zE\tilde k(\tilde x, z, E) = \epsilon \tilde k(\tilde x, z, E)
\label{xtildedot2}
\end{equation}
with $\tilde k$ of class $C^\infty$. Moreover, if $k$ has property $\mathcal{F}$ in $z$, then $\tilde k$ is flat is $z$; in particular,  $\tilde k(\tilde x, 0, E)=0$.
\end{prop}

Note that \eqref{r1a}--\eqref{r3a} has  $\dot x=0$ on the invariant plane $E=0$,but not on the invariant plane $z=0$.  On the other hand, the system \eqref{xtildedot2}, \eqref{r2a}--\eqref{r3a} has  $\dot{\tilde x}=0$ on both invariant planes.  This alone is easy to accomplish; however, if $k$ has property $\mathcal{F}$ in $z$ and nontrivial dependence on $\epsilon$, a careful change of coordinates is needed to yield a $\tilde k$ that is flat is $z$.

\begin{proof}
The equation $\dot x=\alpha_{\tilde x}\dot{\tilde x}+\alpha_E\dot E$ yields
\begin{equation}
\dot{\tilde x}=E\frac{k(\alpha(\tilde x,\epsilon,E),z,\epsilon)-\alpha_E(\tilde x,\epsilon,E)}{\alpha_{\tilde x}}, \quad\epsilon=zE.
\label{xtildedot}
\end{equation}
Now $\alpha_{\tilde x}(\tilde x, \epsilon, E)$ solves the linear initial value problem
$$
u_E(\tilde x, \epsilon, E)=k_x(\alpha(\tilde x,\epsilon,E),0,\epsilon)u, \quad u(\tilde x,\epsilon,0)=1.
$$
Therefore $\alpha_{\tilde x}(\tilde x, \epsilon,E)$, the denominator of the fraction in \eqref{xtildedot}, is nonzero for all $(\tilde x, \epsilon,E)$.  For $z=0$, the numerator of the fraction in \eqref{xtildedot} is $k(\alpha(\tilde x,0,E),0,0)-\alpha_E(\tilde x,0,E)$, which equals 0 by \eqref{alpha}.  Hence \eqref{xtildedot} can be rewritten as \eqref{xtildedot2}.

From \eqref{alpha} we have that $k(\alpha(\tilde x,\epsilon,E),0,\epsilon)-\alpha_E(\tilde x,\epsilon,E)\equiv 0$.  Using the equation one easily checks that if $k$ has property $\mathcal{F}$ in $z$, then the numerator of the fraction in \eqref{xtildedot}, with $\epsilon=zE$, is flat in $z$.  Therefore $\tilde k$ is flat is $z$.
\end{proof}

\begin{prop}
\label{prop:normalform}
Let $N\ge 1$.  Then (1) a $C^\infty$ coordinate change $\bar x=\eta_N(x,z,E)$ brings the system \eqref{r1a}--\eqref{r3a} in the form
\begin{align}
\dot{\bar x} &= \epsilon a(\bar x,\epsilon)+\epsilon^N b(\bar x,z,E), \label{r1b} \\
\dot z &= -z, \label{r2b} \\
\dot E &=E, \label{r3b}
\end{align}
with $a$ and $b$ of class $C^\infty$.
Moreover, (2) if $k$ in \eqref{r1a} has property  $\mathcal{F}$ in $z$, then $a=0$ and $b$ is flat in $z$.
\end{prop}

\begin{proof}
We first prove (1) by induction on $N$.  The case $N=1$ is just Proposition \ref{prop:normalform1}, with $a=0$ and $b=\tilde k$.  Now suppose $N\ge2$ and we have a system of the form
\begin{align}
\dot{y} &= \epsilon p(y,\epsilon)+\epsilon^{N-1} q(y,z,E), \label{r1c} \\
\dot z &= -z, \label{r2c} \\
\dot E &=E, \label{r3c}
\end{align}
with $p$ and $q$ of class $C^\infty$.   We decompose $q$ as
\begin{equation}
\label{q}
q(y,z,E)=q_0(y)+q_1(y,z)+q_2(y,E)+\mathcal{O}(\epsilon),
\end{equation}
with $q_1=\mathcal{O}(z)$ and $q_2=\mathcal{O}(E)$.  Now write
\begin{equation}
\label{xyalphabeta}
\bar x = y+\epsilon^{N-1}(\beta(y,z)+\gamma(y,E)),
\end{equation}
where $\beta$ and $\gamma$ are yet to be specified.  Then
\begin{align}
\dot{\bar x}&=\dot{y}+\epsilon^{N-1}(-\beta_z(y,z)z+\gamma_E(y,E)E)+\mathcal{O}(\epsilon^N) \label{dotbarx1}  \\
&=\dot{y}+\epsilon^{N-1}(-\beta_z(\bar x,z)z+\gamma_E(\bar x,E)E)+\mathcal{O}(\epsilon^N) \label{dotbarx2}
\end{align}
because $y=\bar x + \mathcal{O}(\epsilon)$.  On the other hand, from \eqref{r1c}, \eqref{xyalphabeta}, and \eqref{q},
\begin{align}
\dot y &=\epsilon p(\bar x,\epsilon)+\epsilon^{N-1} q(\bar x,z,E)+\mathcal{O}(\epsilon^N) \label{doty1} \\
&=\epsilon p(\bar x,\epsilon)+\epsilon^{N-1} (q_0(\bar x)+q_1(\bar x,z)+q_2(\bar x,E))+\mathcal{O}(\epsilon^N).  \label{doty2}
\end{align}
Substituting \eqref{doty2} into \eqref{dotbarx2}, we obtain
$$
\dot{\bar x}=(\epsilon p + \epsilon^{N-1}q_0)+\epsilon^{N-1} (q_1-z\beta_z) +\epsilon^{N-1} (q_2+E\gamma_E)+\mathcal{O}(\epsilon^N).
$$
Choosing $\beta=\int\frac{q_1}{z}\,dz$ and $\gamma=-\int\frac{q_2}{E}\,dE$, we obtain \eqref{r1b}.

To prove (2), we assume $k$ in \eqref{r1a} has property  $\mathcal{F}$ in $z$, and show by induction on $N$ that there is a $C^\infty$ coordinate change $\bar x=\eta_N(x,z,E)$ that converts  \eqref{r1a}--\eqref{r3a} to
\begin{align}
\dot{\bar x} &= \epsilon^N b(\bar x,z,E), \label{r1bf} \\
\dot z &= -z, \label{r2bf} \\
\dot E &=E, \label{r3bf}
\end{align}
where $b$ is $C^\infty$ and flat in $z$.  The case $N=1$ is given by the comment after formula \eqref{xtildedot2}, with $b=\tilde k$.  Now suppose $N\ge2$ and we have a system of the form
\begin{align}
\dot{y} &=\epsilon^{N-1} q(y,z,E), \label{r1cf} \\
\dot z &= -z, \label{r2cf} \\
\dot E &=E, \label{r3cf}
\end{align}
where $q$ is $C^\infty$ and flat in $z$.   In this case we can write
$$
q(y,z,E)=q_1(y,z)+\epsilon q_3(y,z,\epsilon),
$$
with $q_1$ and $q_3$ flat in $z$.  Now write
\begin{equation}
\label{xyalphabetaf}
\bar x = y+\epsilon^{N-1}\beta(y,z),
\end{equation}
with $\beta$ not yet specified but flat in $z$.  Then
\begin{align}
\dot{\bar x}&=\dot{y}-\epsilon^{N-1}\beta_z(y,z)z+\epsilon^Nq_4(y,z,E)
=\dot{y}-\epsilon^{N-1}\beta_z(\bar x,z)z+\epsilon^Nq_5(\bar x,z,E), \label{dotbarx2f}
\end{align}
with $q_4$ and $q_5$ flat in $z$.  On the other hand, from \eqref{r1cf}, \eqref{xyalphabetaf}, and \eqref{q},
\begin{align}
\dot y &=\epsilon^{N-1} q(\bar x,z,E)+\epsilon^Nq_6(\bar x,z,E)
=\epsilon^{N-1}q_1(\bar x,z)+\epsilon^Nq_7(\bar x,z,E),  \label{doty2f}
\end{align}
with $q_6$ and $q_7$ flat in $z$. Substituting \eqref{doty2f} into \eqref{dotbarx2f}, we obtain
$$
\dot{\bar x}=\epsilon^{N-1} (q_1-z\beta_z) +\epsilon^Nq_8(\bar x,z,E),
$$
with $q_8$ flat in $z$. Choosing $\beta=\int\frac{q_1}{z}\,dz$, which is flat in $z$ as required, we obtain \eqref{r1bf}.
\end{proof}

\subsection{Integration}

For simplicity we take $z_0=1$.  Thus to define $P_1^a: S_0^a \to S_1^a$ for $\epsilon>0$, we wish to integrate \eqref{r1a}--\eqref{r3a} from an initial point $(x,z,E)=(x_0,1,\epsilon)$ until $E=E_1$, i.e., to a final point $(x_1,\frac{\epsilon}{E_1},E_1)$.

We consider \eqref{r1b}--\eqref{r3b} with $N$ replaced by $N+2$, $N\ge 1$:
\begin{align}
\dot{\bar x} &= \epsilon a(\bar x,\epsilon)+\epsilon^{N+2} b(\bar x,z,E), \label{r1h} \\
\dot z &= -z, \label{r2h} \\
\dot E &=E. \label{r3h}
\end{align}
In these coordinates, the initial and final points are given by $(\bar x_0,1,\epsilon)$ and $(\bar x_1,\frac{\epsilon}{E_1},E_1)$, where $\bar x_i$ depends smoothly on $(x_i,\epsilon)$, as shown in Proposition~\ref{prop:normalform}.

In \eqref{r1h}--\eqref{r3h} we  make the additional change of variables $\bar z=zE\log z=\epsilon\log z$ and $\bar E=zE\log E=\epsilon\log E$.  In addition, we use $t$ to denote time in \eqref{r1a}--\eqref{r3a} or \eqref{r1h}--\eqref{r3h}, we introduce the new time $\tau=\frac{t}{\epsilon}$, and we use prime to denote derivative with respect to $\tau$. We obtain
\begin{align}
\bar x^\prime &= a(\bar x,\epsilon)+\epsilon^{N+1} b(\bar x,e^{\bar z/\epsilon},e^{\bar E/\epsilon}), \label{r1d} \\
\bar z^\prime &= -1, \label{r2d} \\
\bar E^\prime &=1. \label{r3d}
\end{align}
We need to integrate this system from $(\bar x,\bar z,\bar E)=(\bar x_0,0,\bar E_0)$, with $\bar E_0=\epsilon \log \epsilon$, to $(\bar x,\bar z,\bar E)=(\bar x_1,\bar z_1,\bar E_1)$, with $\bar z_1=\epsilon \log \frac{\epsilon}{E_1}$ and $\bar E_1=\epsilon \log E_1$.  See Figure \ref{fig:domain}.
The time of integration is $\tau=\epsilon \log E_1-\epsilon \log \epsilon=\mathcal{O}(\epsilon\log\epsilon)$.

\begin{figure}[htb]
\centerline{\includegraph[width=3.5in]{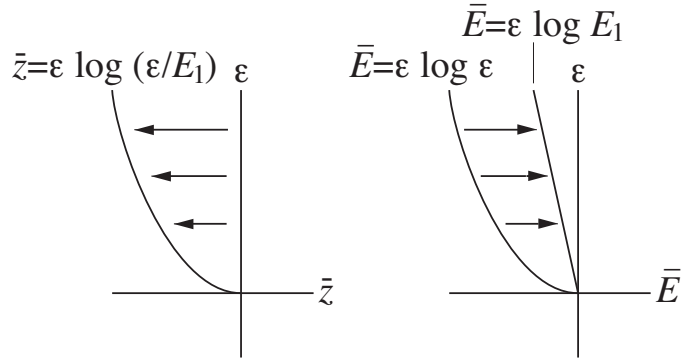}}
\caption{Change of $\bar z$ and $\bar E$ along orbits of \eqref{r1d}--\eqref{r3d} as a function of $\epsilon$.  In the second graph it is assumed that $E_1<1$.}
\label{fig:domain}
\end{figure}

Since $\epsilon$ is constant on solutions of \eqref{r1d}--\eqref{r3d}, we regard it as a parameter.  From the previous paragraph, we need only consider \eqref{r1d}--\eqref{r3d} on the region
$$
\mathcal{D}=\{(\bar x,\bar z,\bar E,\epsilon) : \epsilon \log \frac{\epsilon}{E_1} \le \bar z \le 0, \; \epsilon \log \epsilon \le \bar E \le \epsilon \log E_1 \}.
$$
One can check that on $\mathcal{D}$, the system \eqref{r1d}--\eqref{r3d} is of class $C^{N}$.  In particular, for a given $\bar x^*$, any mixed partial derivative of order up to $N$ of the function $\epsilon^{N+1} b(\bar x,e^{\bar z/\epsilon},e^{\bar E/\epsilon})$ approaches 0 as $(\bar x,\bar z,\bar E,\epsilon) \to (\bar x^*,0,0,0)$ within $\mathcal{D}$.  For example
\begin{multline*}
\frac{\partial^{N} \,\,}{\partial \epsilon^{N}}\epsilon^{N+1} b(\bar x,e^{\bar z/\epsilon},e^{\bar E/\epsilon})=\epsilon^{N+1}D_3b(\bar x,e^{\bar z/\epsilon},e^{\bar E/\epsilon})\left(-\frac{\bar E}{\epsilon^2}\right)^{N} +\cdots \\
=(-1)^{N}\frac{\bar E^{N}}{\epsilon^{N-1}}D_3b(\bar x,e^{\bar z/\epsilon},e^{\bar E/\epsilon})+\cdots .
\end{multline*}
Within $\mathcal{D}$ the quotient
$
\frac{\bar E^{N}}{\epsilon^{N-1}}
$
approaches 0 as $(\bar x,\bar z,\bar E,\epsilon) \to (\bar x^*,0,0,0)$.

The solution of \eqref{r1d}--\eqref{r3d} with initial condition $(\bar x,\bar z,\bar E)=(\bar x_0,0,\bar E_0)$ at $\tau=0$ has $\bar x$-coordinate $\bar x=\phi(\bar x_0,\bar E_0,\epsilon,\tau)$, where $\phi$ is $C^N$ as long as the solution remains in $\mathcal D$.  Thus
\begin{equation}
\label{phi}
\bar x_1=\phi(\bar x_0,\bar E_0,\epsilon,\epsilon \log E_1-\epsilon \log \epsilon)=\phi(\bar x_0,\epsilon \log \epsilon,\epsilon,\epsilon \log E_1-\epsilon \log \epsilon).
\end{equation}
More compactly, $\bar x_1$ is a $C^N$ function of $(\bar x_0,\epsilon,v)$ with $v=\epsilon \log \epsilon$.  Since it follows from Proposition \ref{prop:normalform} that $\bar x_0$ is $C^\infty$ function of $(x_0,\epsilon)$ and $x_1$ is a $C^\infty$ function of $(\bar x_1,\epsilon)$, we see that $x_1$ is a
$C^N$ function of $(x_0,\epsilon,\epsilon\log\epsilon)$.  This proves conclusion (1) of Proposition~\ref{cornerprop1}.

To prove conclusion (3) of Proposition \ref{cornerprop1}, we note that if $\h(x,z)$ 
has property $\mathcal{F}$ in $z$, then so does $k=1/h$, so from  Proposition \ref{prop:normalform}, in \eqref{r1d}--\eqref{r3d} we can take $a=0$.  In consequence, after an integration time of $\mathcal{O}(\epsilon\log\epsilon)$, $\bar x$ cannot have changed more than an amount $\mathcal{O}(\epsilon^{N+2}\log\epsilon)$, so the formula for $\phi$ in \eqref{phi} must be of the form $\bar x_0 + \mathcal{O}(\epsilon^{N+2}\log\epsilon)$. Hence $\phi$ is $o(\epsilon^{N+1})$.   Therefore, when we write $\bar x_1$ as a $C^N$ function of $(\bar x_0,\epsilon,v)$ with $v=\epsilon \log \epsilon$, the $N$th degree Taylor polynomial in $(\epsilon,v)$,
$$
\bar x_1\sim \bar x_0 + \sum_{1\le k+l\le N} a_{kl}(\bar x_0)\epsilon^kv^l,
$$
has all $a_{kl}=0$. Since $\bar x_0$ is $C^\infty$ function of $(x_0,\epsilon)$ and $x_1$ is a $C^\infty$ function of $(\bar x_1,\epsilon)$, it follows that when we write $x_1$ as a $C^N$ function of $(x_0,\epsilon,v)$, the $N$th degree Taylor polynomial in $(\epsilon,v)$ has no terms involving $v$.

\section{Finite differentiability\label{sec:finite}}

Suppose $f$ and $g$ in \eqref{eq3}--\eqref{eq4}, and hence $h$ in \eqref{eq5}--\eqref{eq6} and $k$ in \eqref{r1a}--\eqref{r3a}, are $C^r$.  The coordinate changes of Section \ref{sec:local} then result in a reduction of differentiability.

In particular, $\tilde k$ in \eqref{xtildedot2} is $C^{r-1}$. Now in \eqref{r1b} with $N=1$, $a=0$ and $b=\tilde k$, so for $N=1$ we have $a$ and $b$ of class $C^{r-1}$.

In the induction proof of conclusion (1) of Proposition \ref{prop:normalform}, suppose $p$ and $q$ in \eqref{r1c} are of class $C^s$.  In \eqref{q} the term $\mathcal{O}(\epsilon)$ is actually $\epsilon q_3(y,z,E)$ where $q_3$ is $C^{s-2}$.  The functions $q_1$ and $q_2$ are of class $C^s$, so $\beta$ and $\gamma$ constructed in the proof are of class $C^{s-1}$.  Thus when we replace $N-1$ by $N$ in \eqref{r1c}, $p$ and $q$ will be of class $C^{s-2}$.

It follows that when we produce the normal form \eqref{r1h}--\eqref{r3h}, starting from $f$ and $g$ of class $C^r$, $a$ and $b$ are of class of $C^{r-2(N+1)-1}=C^{r-2N-3}$.  In order to obtain a flow on $\mathcal D$ of class $C^N$ as described in Section \ref{sec:local}, $a$ and $b$ must be of class $C^N$, so we need $r-2N-3\ge N$, or $r\ge 3N+3$.

Thus the $C^N$ function of Propositions \ref{cornerprop1} and \ref{cornerprop2} exists provided $f$ and $g$ are $C^r$,  $r\ge 3N+3$.  Hence in Theorem \ref{th:main}, if $f$ and $g$ are $C^r$,  $r\ge 3N+3$, then conclusions (1) and (2) hold, with $\tilde p$ of class $C^N$.  In particular, $\tilde p$ is $C^1$ if $r\ge6$.

It is somewhat tedious to adapt conclusion~(3) of Theorem~\ref{th:main} to the case of finitely smooth systems, so we do not address this issue here.

\section{Example with logarithmic terms in the  return map}\label{sec:ex}

We saw in Section \ref{sec:local} that the logarithmic terms in Proposition \ref{cornerprop1} appear because an integration requires time $\tau=\mathcal{O}(\epsilon\log\epsilon)$.  Of course one might wonder whether or not such terms could cancel in the end.  In this section we show that in Example \eqref{ex1}--\eqref{ex2}, for any fixed small $\alpha\neq0$, logarithmic terms really do arise in the expansion with respect to $\epsilon$ of the return map.

We rewrite \eqref{ex1}--\eqref{ex2} as
\begin{equation} \label{ex1alt}
\frac{dz}{dx}=\frac{(x+\alpha z)z^2}{\epsilon}.
\end{equation}
For $\alpha=0$, a family of solutions is given by
\begin{equation} \label{z0}
    z_0(x) = \frac{2\epsilon}{2\epsilon + x_0^2 - x^2}.
\end{equation}
The solutions are parameterized by $x_0<0$ in such a way that $z_0(x_0)=1$. Since $z_0(-x_0)=1$ as well, the return map for $\alpha=0$ on the line $z=1$  is simply the mapping $x_0\to -x_0$.

For small $\alpha$, we write the solution of \eqref{ex1alt} with $z(x_0)=1$ as
\begin{equation} \label{zexp}
z(x) = z_0(x) + \alpha z_1(x) + O(\alpha^2).
\end{equation}
Substitution into \eqref{ex1alt}  yields
\[
    \epsilon\frac{dz_1}{dx} = 2xz_0z_1 + z_0^{3},\qquad z_1(x_0)=0.
\]
The solution of this linear differential equation is
\begin{equation} \label{z1}
    z_1(x) = \frac{8\epsilon^{2}}{(2\epsilon+x_0^2-x^2)^2}\int_{x_0}^x\left(2\epsilon+x_0^2-s^2\right)^{-1}ds.
\end{equation}
Since the return map on the line $z=1$ is smooth in~$(x_0,\alpha,\epsilon,\epsilon\log\epsilon)$ (see Remark~\ref{rem}), we write it as a series in~$\alpha$ as well: $x_0\mapsto -x_0 + c(x_0,\epsilon)\alpha + O(\alpha^2)$.
Writing $z(-x_0 + c(x_0,\epsilon)\alpha + O(\alpha^2))=1$ and using \eqref{zexp}, we find that
\[
    z_0(-x_0) + c\alpha z_0'(-x_0) + z_1(-x_0)\alpha + O(\alpha^2) = 1.
\]
Therefore
\[
c(x_0,\epsilon) = -\frac{z_1(-x_0)}{z_0'(-x_0)}.
\]
Using \eqref{z0} and \eqref{z1}, we obtain
\begin{equation}\label{c}
    c(x_0,\epsilon) =  \frac{2\epsilon}{x_0}\int_{x_0}^{-x_0}(2\epsilon + x_0^2-s^2)^{-1}ds.
\end{equation}

\begin{prop}\label{prop:c}
    Expression \eqref{c} is only finitely smooth w.r.t.~$(x_0,\epsilon)$.  More precisely,  there is an analytic function $q(x_0,\epsilon)\not=0$ such that $c(x_0,\epsilon)-q(x_0,\epsilon)\epsilon\log\epsilon$ is analytic in $(x_0,\epsilon)$.
\end{prop}

\begin{proof}
  In \eqref{c} we let $\tilde{\epsilon} = \epsilon/x_0^2$ and $\tilde{c}= (2\epsilon)^{-1} x_0^2 c$, and we make the substitution $s=x_0u$.  We obtain    \[
        \tilde{c}=-2\int_{0}^{1}(2\tilde{\epsilon} + 1-u^2)^{ -1}du= -\int_{0}^{1}\frac{A}{\sqrt{1+2\tilde{\epsilon}}-u}du -\int_{0}^{1}\frac{A}{\sqrt{1+2\tilde{\epsilon}}+u}du,
    \]
with
$$
A=\frac{1}{\sqrt{1+2\tilde{\epsilon}}}.
$$
The second integral is analytic in $\epsilon$ near $\epsilon=0$.  The first is
 $$
- \int_0^1 \frac{A}{\sqrt{1+2\tilde{\epsilon}}-u}du =A\left(\log(\sqrt{1+2\tilde\epsilon}-1)-\log\sqrt{1+2\tilde{\epsilon}} \right) = A\log\tilde{\epsilon}+ O(\tilde{\epsilon}),$$
 where the  term  $O(\tilde{\epsilon})$ is analytic.  Therefore
 $$
 \tilde c = \left(1+r(\tilde\epsilon)\right)\log \tilde\epsilon + s(\tilde \epsilon)
 $$
where $r$ and $s$ are analytic.  The result follows by substituting $\tilde{\epsilon} = \epsilon/x_0^2$ and $\tilde{c}= (2\epsilon)^{-1} x_0^2 c$.
 \end{proof}

Since the return map of Example \eqref{ex1}--\eqref{ex2} is given by $x_0\mapsto -x_0 + c\alpha + O(\alpha^2)$, the finite smoothness of $c(x_0,\epsilon)$ expressed in Proposition~\ref{prop:c} implies the finite smoothness with respect to~$\epsilon$ of the return map for small $\alpha$, which we intended to show.

\section{Acknowledgements}

This work was partly supported by FWO project G093910 (De Maesschalck) and NSF grant  DMS-1211707 (Schecter).

\end{document}